\documentclass{amsart}

\usepackage{hyperref}
\usepackage{amsmath}
\usepackage{amssymb}
\usepackage[capitalize]{cleveref}
\usepackage{amsthm}
\usepackage{todonotes}
\usepackage[backend=biber,
           style=alphabetic,
            url=false
           ]{biblatex}

\usepackage{color}
\usepackage{setspace}
\usepackage{tikz-qtree}
\usepackage{thmtools}
\usepackage{csquotes}
\newtheorem{theorem}{Theorem}
\newtheorem{lemma}[theorem]{Lemma}

\newtheorem{corollary}[theorem]{Corollary}

\theoremstyle{definition}
\newtheorem{definition}[theorem]{Definition}

\theoremstyle{remark}

\newcommand{\A}{\mathcal{A}}

\newcommand{\mc}[1]{\mathcal{#1}}

\newcommand{\ra}{\rightarrow}
\newcommand{\lr}{\leftrightarrow}
\newcommand{\La}{\Leftarrow}
\newcommand{\Ra}{\Rightarrow}
\newcommand{\peql}{\preceq^{\mc L}}

\newcommand{\peqai}{\preceq^{\mc A_i}}

\DeclareMathOperator{\restrict}{\upharpoonright}

\renewcommand{\restrict}{\upharpoonright}
\renewcommand{\phi}{\varphi}

\usepackage{soul}

\newbox\gnBoxA
\newdimen\gnCornerHgt
\setbox\gnBoxA=\hbox{$\ulcorner$}
\global\gnCornerHgt=\ht\gnBoxA
\newdimen\gnArgHgt
\def\gnmb #1{%
\setbox\gnBoxA=\hbox{$#1$}%
\gnArgHgt=\ht\gnBoxA%
\ifnum     \gnArgHgt<\gnCornerHgt \gnArgHgt=0pt%
\else \advance \gnArgHgt by -\gnCornerHgt%
\fi \raise\gnArgHgt\hbox{$\ulcorner$} \box\gnBoxA %
\raise\gnArgHgt\hbox{$\urcorner$}}

\newcommand{\goed}[1]{\ulcorner #1\urcorner}
\newcommand{\PA}{\mathsf{PA}}
\newcommand{\C}{\psi}
\newcommand{\Prm}[1]{\mathrm{Prv}_{#1}}
\newcommand{\kommentar}[1]{}
\newcommand{\F}{\theta}
\newcommand{\E}{\chi}

\title{Feferman's completeness theorem}
\author[Pakhomov]{Fedor Pakhomov}
\author[Rathjen]{Michael Rathjen}
\author[Rossegger]{Dino Rossegger}
\address[Pakhomov]{Vakgroep Wiskunde: Analyse, Logica En Discrete Wiskunde, 
    Universiteit Gent and 
    Steklov Mathematical Institute of Russian Academy of Sciences, Moscow}
    \email{\href{mailto:fedor.pakhomov@ugent.be}{fedor.pakhomov@ugent.be}}
    \urladdr{\url{https://homepage.mi-ras.ru/~pakhfn/}}

\address[Rathjen]{School of Mathematics, University of Leeds}
\email{\href{mailto:m.rathjen@leeds.ac.uk}{m.rathjen@leeds.ac.uk}}
\urladdr{\url{http://www1.maths.leeds.ac.uk/~rathjen/rathjen.html}}

\address[Rossegger]{Institut f\"ur Diskrete Mathematik und Geometrie, 
  Technische Universit\"at Wien
  }
\email{\href{mailto:dino.rossegger@tuwien.ac.at}{dino.rossegger@tuwien.ac.at}}
\urladdr{\url{https://drossegger.github.io/}}

\date{\today}
\thanks{The work of Pakhomov was funded by the FWO grant G0F8421N. 
The work of Rathjen was supported by the John Templeton Foundation Grant ID
60842 and by the European Union (ERC Advanced Grant, C-FORS: Construction in the Formal Sciences, awarded to {\O}ystein Linnebo, project number 101054836). The work of Rossegger was supported by the European Union's Horizon 2020 Research and Innovation Programme under the Marie Sk\l{}odowska-Curie grant agreement No. 101026834 — ACOSE}
\begin{document}
\maketitle
\begin{abstract}
    Feferman proved in 1962~\cite{feferman1962} that any arithmetical theorem is
    a consequence of a suitable transfinite iteration of full uniform reflection
    of $\mathsf{PA}$. This result is commonly known as Feferman's completeness
    theorem. The purpose of this paper is twofold. On the one hand this
    is an expository paper, giving two new proofs of Feferman's completeness
    theorem that, we hope, shed light on this mysterious and often overlooked
    result. 
    On the other hand, we combine one of our proofs with results from computable
    structure theory due to Ash and Knight to give sharp bounds on the order
    types of well-orders necessary to attain the completeness for levels of the arithmetical hierarchy.
\end{abstract}
\section{Introduction}
By G\"odel's Second Incompleteness Theorem no consistent extension of
$\mathsf{PA}$ with a computably enumerable axiomatization can prove its own
consistency. Thus, whenever we extend a sound arithmetical theory $T$ to a sound
theory $T'$ that proves the consistency of $T$, it is guaranteed that
$T'$ is stronger than $T$.
 The idea of using this
phenomenon to attain transfinite sequences of theories of ascending strength
 goes back to Alan Turing~\cite{turing1939}. In modern terms, Turing
defined transfinite iterations of local reflection
$\mathsf{Rfn}$ roughly as 
\[\mathsf{Rfn}^0(T)=T\qquad\mathsf{Rfn}^{\alpha+1}(T)=T+\mathsf{Rfn}(\mathsf{Rfn}^\alpha(T))\qquad \mathsf{Rfn}^\lambda(T)=\bigcup\limits_{\alpha<\lambda}\mathsf{Rfn}^\alpha(T).\]
Here, for a theory $T$ with a fixed c.e.\ axiomatization, $\mathsf{Rfn}(T)$ is the collection of all principles
\[\mathsf{Prv}_T(\gnmb{\varphi})\to \varphi\text{, where $\varphi$ is any arithmetical sentence.}\]
A subtle feature of this definition is that at each step of the inductive definition, one has to pick a c.e.\ axiomatization of $\mathsf{Rfn}^\alpha(T)$.
Thus $\alpha$ has to be a computable ordinal with a fixed computable
presentations, hence the theories $\mathsf{Rfn}^\alpha(T)$ depend not only on the order-type of
$\alpha$ but also on the particular computable presentation. As was proven already by Turing, for any true $\Pi_1$-sentence $\phi$ there is a
computable presentation $\alpha$ of the ordinal $\omega+1$, in fact an ordinal
notation in Kleene's $\mathcal{O}$, such that
$\mathsf{Rfn}^\alpha(\mathsf{PA})$ proves $\phi$. Of course, there cannot be a single
computable presentation $\alpha$ such that $\mathsf{Rfn}^\alpha(\mathsf{PA})$ proves 
every true $\Pi_1$ sentence, as this would mean that there is a consistent c.e.\
extension of $\mathsf{PA}$ proving all true $\Pi_1$ statements. Thus, the
construction of theories $\mathsf{Rfn}^\alpha(\mathsf{PA})$ has an intensional character: Their consequences depend not only on the order type of $\alpha$ but also on the choice of a particular computable presentation.

Turing also conjectured that for every arithmetical theorem $\phi$, there is a computable
presentation of an ordinal $\alpha$ such that $\mathsf{Rfn}^\alpha(\mathsf{PA})$ proves $\phi$. Later,
Feferman~\cite{feferman1962} refuted this conjecture. However, he considered a variant of
Turing's progression, where local reflections $\mathsf{Rfn}(T)$ are replaced with uniform reflections $\mathsf{RFN}(T)$:
\[\forall x \big(\mathsf{Prv}_T(\gnmb{\varphi(\dot x)})\to \varphi(x)\big)\text{, where $\varphi(x)$ is any arithmetical formula.}\]
For this variant, Feferman proved that Turing's conjecture holds,
i.e., any true arithmetical sentence is provable in
$\mathsf{RFN}^\alpha(\mathsf{PA})$. This result is often referred to as
\emph{Feferman's completeness theorem}. Feferman's construction gives
$\omega^{{\omega}^{\omega+1}}$ as an upper bound on the order type of the ordinal
notations along which one needs to iterate reflection to obtain a proof of
$\phi$.

The presence of an upper bound in Feferman's result indicates that it is a negative result: It essentially shows that proof-theoretic strength of an
arithmetical sentence $\varphi$ couldn't be appropriately measured by the least
order-type of ordinal notation, iteration of reflection along which proves
$\varphi$. However, when restricted to natural ordinal notation systems the length
of iterations of reflection necessary to obtain a particular theory does reflect
the proof-theoretic strength of the theory. This phenomenon was explored and used
for the applications to ordinal analysis \cite{schmerl1979fine,beklemishev2003proof}. 
We furthermore note that as was shown by the first author and James Walsh
\cite{pakhomov2021reflection}, for systems of second-order arithmetic and their
 $\Pi^1_1$-consequences, the strength of the systems
can be appropriately measured in terms of the length of iterations of
$\Pi^1_1$-reflection, which is a measure closely connected with the $\Pi^1_1$-proof
theoretic ordinals of theories. 

Feferman's original proof of his completeness theorem is not widely known, perhaps because of its considerable technical difficulty. One feature that Feferman's proof shares with Turing's (rather simple) proof of the $\Pi^0_1$-completeness of consistency progressions is the employment of 
non-standard definitions of the axioms of theories. These non-standard definitions are engineered so that we know them to be an axiomatization of a ``natural'' theory only if we know a certain sentence $\varphi$ to be true. In Feferman's proof, they appear in ever more intertwined ways. 
 Very artfully, he applies increasingly more entangled versions of the recursion theorem to generate non-standard definitions of theories (see \cite{franzen2004} for further details).
 Franz{\'e}n \cite[p. 387]{franzen2004} expressed this as follows: 
 \begin{displayquote}{\em The completeness theorem [i.e., Feferman's] can be seen as a dramatic illustration of the role of intensionality in logic.}
 \end{displayquote}

In this paper, we present three new proofs of Feferman's completeness theorem.
The first is a simple proof of this theorem using results not available to
Feferman at the time such as the conservativity of $\mathsf{ACA_0}$ over
$\mathsf{PA}$. The second proof is an alternative proof using techniques that
Feferman could have had access to, such as Schütte's completeness theorem for
$\omega$-logic. At last, we combine our first proof with results of Ash and
Knight~\cite{ash1990} in computable structure theory to obtain precise bounds on
the order types appearing in Feferman's theorem. 

\subsection{Structure of the paper}
The paper is structured as follows. First, in \cref{Feferman_section} we give a
fully self-contained presentation of our simple proof of Feferman's completeness
theorem.
The proof is split into three well-understood steps:
\begin{enumerate}
    \item For any arithmetical sentence $\varphi$ there is a computable order $\mc L$ such that \[\mathsf{ACA}_0\vdash \mathsf{WO}(\mc L)\mathrel{\leftrightarrow} \varphi\;\;\;\text{and}\;\;\; \mathsf{ACA}_0\vdash \mathsf{LO}(\mc L).\] 
    
    \item For any computable linear order $\mc L$ such that
        $\mathsf{ACA}_0\vdash \mathsf{LO}(\mc L)$ the theories
        $\mathsf{ACA}_0+\mathsf{WO}(\mc L)$ and $\mathsf{PA}+\mathsf{TI}(\mc L)$ prove exactly the same first-order sentences.

    \item As we prove in Lemma \ref{reflection_implies_TI}
    \[\mathsf{RFN}^{\mathcal{L}+1}(\mathsf{PA})\vdash \mathsf{TI}(\mc L).\]
    
\end{enumerate}

The ideas of the proofs of these components are the following:
\begin{enumerate}
    \item     This is a particular case of $\mathsf{ACA}_0$-provable $\Pi^1_1$-completeness of well-orders. For the convenience of readers, we give a standard proof of the fact (Lemma \ref{arith-completeness}).

\item     We prove this in Lemma \ref{sen2wo} by showing that the extension of any model of $\mathsf{PA}+\mathsf{LO}(\mc L)+\mathsf{TI}(\mc L)$ by the second-order universe consisting of first-order definable sets is a model of  $\mathsf{ACA}_0+\mathsf{WO}(\mc L)$.

\item Like most other facts about the iterations of reflection principles, this
    fact is straightforward to prove using L\"ob's Theorem.
\end{enumerate}
We note that neither of these facts is actually new. The first goes back to
Kleene and the formalized version can be found in \cite{simpson2009}. The second fact is a triviality, while the third one is folklore. 

In \cref{sec:history} we give a proof of Feferman's completeness theorem using
techniques available at the time and in \cref{upperbounds_sect,lowerbounds_sect} we obtain sharp ordinal bounds for Feferman's
completeness theorem. 

\begin{restatable}{theorem}{upperbound}\label{upperbound} For every true $\Pi_{2n+1}$ sentence $\varphi$ there exists a computable $\mathsf{PA}$-verifiable linear order $\mc L\cong \omega^n+1$ such that
  \[\mathsf{RFN}^{\mc L}(\mathsf{PA})\vdash \varphi.\]
\end{restatable}

To prove Theorem \ref{upperbound} we replace the use of $\Pi^1_1$-completeness
of computable well-orders with a sharper theorem of Ash and Knight
\cite{ash1990}, who proved that the set of indices of computable orderings of order-type $\omega^n$ is $\Pi^0_{2n+1}$-hard.

We complement Theorem \ref{upperbound} with the following result proving that the upper bound on order-types is in fact sharp:

\begin{restatable}{theorem}{lowerbound}\label{lowerbound} There exists a true
    $\Pi_{2n}$-sentence $\varphi$ such that for any computable
    $\mathsf{PA}$-verifiable linear order $\mc L$ of order-type at most $\omega^n$
  \[\mathsf{RFN}^{\mc L}(\mathsf{PA})\nvdash \varphi.\]
\end{restatable}

The key point in the proof of Theorem \ref{lowerbound} is that for any ordinal $\alpha<\omega^n$ there is a $\Sigma_{2n}$ definition of the property of an index of a computable order to be an isomorphic copy of $\alpha$.

We remark that in our main results, we are dealing with reflection iterated along
arbitrary computable well-orders, while classically Feferman and Turing
considered iterations along ordinal notations from Kleene's $\mc O$. In fact the
results are sensitive to this distinction and in \cref{Kleene_O_Section} we establish $\omega^{n+1}+1$ as an upper bound on the order types of elements of $\mc O$ necessary to attain arbitrary true $\Pi_{2n+1}$-sentences as the consequences of iterated full uniform reflection.
\subsection{Acknowledgements} The first author would like to thank Stella Moon
whose questions led him to start investigations on Feferman's theorem. 
He would also like to thank Nikolay Bazhenov for pointing him to the work of Ash and
Knight. We would like to thank Andrey Frolov for sharing with us his neat
alternative proof of the theorem of Ash and Knight. Finally, we would like to thank the anonymous referee, as well as Mateusz Łełyk, Patryk Szlufik, and Patrick Lutz for their useful comments on an earlier version of the paper.

\section{Feferman's Completeness Theorem via Kleene's Normal Forms}\label{Feferman_section}
\subsection{Computable orders}
Usually, recursive progressions proceed along some ordinal given by some ordinal
notation system such as Kleene's $\mathcal O$. We deviate from this since we will
work directly with indices for computable orderings. 

A linear ordering $\mc L$ is a pair consisting of a set $L$ and a binary
relation $\preceq^{\mc L}$ satisfying the usual axioms of linear orders.\footnote{Formally, $\preceq^{\mc L}$ is a set of natural numbers, however we may view it as a binary relation using the standard pairing function $\langle i,j\rangle= (i+j)^2+i$.}. Note that $L$ could be recovered from $\preceq^{\mc L}$ as the set of $x$ such that $x\preceq x$.

Formally within first-order arithmetic, we represent computable orders $\mc L$ as natural numbers coding pairs consisting of an index of a $\Pi_1$-set and an index of a $\Sigma_1$-set such that they represent the same set $\preceq^{\mc L}$ and this set is a binary relation satisfying the axioms of linear orders. There is a first-order arithmetical formula $\mathsf{LO}(\mc L)$ expressing that $\mc L$ is indeed a computable linear order in the sense above.  We say that an order $\mc L$ is a $\mathsf{PA}$-verifiable computable linear order if $\mathsf{PA}$ proves $\mathsf{LO}(\mc L)$. Over $\mathsf{ACA}_0$ we formulate the second-order formula $\mathsf{WO}(\mc L)$ expressing that $\mc L$ is a well-order:
\[ \mathsf{LO}(\mc L)\land  \forall Y \big(\exists x(x\in Y\cap \mc L) \to (\exists x\in Y\cap \mc L)(\forall y\in Y\cap \mc L)
(x\preceq^{\mc L} y)\big)\]

The following theorem gives the first part of our proof of Feferman's
completeness theorem.
\begin{lemma}\label{arith-completeness}
    For any arithmetical sentence $\phi$ there exists a $\mathsf{PA}$-verifiable computable linear order
    $\mc L$ such that $\mathsf{ACA_0}\vdash \phi \lr \mathsf{WO}(\mc L)$.
\end{lemma}
This follows immediately from the fact that it is provable in $\mathsf{ACA_0}$ that the set
of indices of computable well-orders is a complete $\Pi^1_1$-set and that every
arithmetical formula is $\Pi^1_1$, see~\cite[Lemma
V.1.8]{simpson2009}.

For the convenience of the reader, we give a proof of Lemma \ref{arith-completeness} using standard techniques.
\begin{proof}
    We claim that there is a formula equivalent to $\lnot \varphi$ over
    $\mathsf{ACA}_0$ that is of the form $(\exists f: \mathbb{N}\to\mathbb{N})\forall x \psi(f,x)$, where $\psi(f,x)$ is a quantifier-free formula whose atomic subformulas are equalities of terms built from $f$ and primitive-recursive formulas.
    
    To prove the claim we consider $\varphi$ to be a formula built using Boolean
    connectives and existential quantifiers. Then we consider the expansion of
    the language of arithmetic by Skolem functions $f_{\chi}(\vec{y})$ for all the subformulas $\chi$ of $\varphi$ of the form
    $\exists x \xi(x,\vec{y})$.  Now, by induction on subformulas
    $\theta(\vec{x})$ of $\varphi$, we define their Skolemized variants
    $\theta'(\vec{x})$ that are quantifier-free formulas in the expanded
    language: the transformation $(\cdot)'$ commutes with Boolean connectives
    and $\chi=(\exists x \xi(x,\vec{y}))'$ is $\xi(f_{\chi}(\vec y),\vec{y})$.
    For each Skolem function $f_\chi$ where $\chi=\exists x \xi(x,\vec{y})$ we have the
    axiom $\xi(x,\vec{y})\to\xi(f_{\chi}(\vec{y}),\vec{y})$ that expresses that
    $f_{\chi}(\vec{y})$ is indeed a Skolem function for $\exists x \,
\xi(x,\vec{y})$. This allows us to transform $\lnot \varphi$ over
$\mathsf{ACA}_0$ to a formula of the form 
\[\exists \vec{f}\forall \vec{x}\ \psi(\vec{f},\vec{x}),\]
where the quantifier $\exists \vec{f}$ quantifies over the candidates for Skolem
functions and the following formula is the universal closure of the conjunction
of the axioms expressing that the $f$'s indeed are designated Skolem functions and
$\lnot \varphi'$. After that, using supplementary primitive-recursive
functions, it is easy to merge multiple multi-variant $f$'s into a single unary $f$ and multiple $\vec{x}$ into a single $x$. Which finishes the proof of the claim.

    Subsequently we transform $\varphi$ from the form 
    \[(\forall f: \mathbb{N}\to\mathbb{N})\exists x\, \lnot \psi(f,x) \quad
        \text{to} \quad
    \forall f\exists x\, \theta(f\upharpoonright x),\]
    where $\theta(y)$ is a quantifier-free formula in the language with primitive-recursive functions that expresses that
    \begin{enumerate}
        \item $y$ is a code of a sequence $(s_0,\ldots,s_{k-1})$ that we treat as a pair consisting of $x=s_0$ and a partial function $f\colon i\mapsto s_{i+1}$,
        \item the domain of $f$ is sufficient to evaluate the validity of the formula $\psi(f,x)$, 
        \item and the formula $\psi(f,x)$ is false.
    \end{enumerate} 

    Hence over $\mathsf{ACA}_0$, $\varphi$ is equivalent to the well-foundedness of the primitive-recursive tree of sequences \[T=\{(s_0,\ldots,s_{k-1})\mid\theta((s_0,\ldots,s_{k-1}))\}.\]
    We construct $\mc L$ as the Kleene-Brouwer order on $T$, i.e., the order
    where each node $(s_0,\ldots,s_{k-1})$ is bigger than all its descendants
    $(s_0,\ldots,s_{k-1},u_k,\ldots,u_{l-1})$ and for all $u_k<u_k'$ all the elements in the subtree of $(s_0,\ldots,s_{k-1},u_k)$ are smaller than the elements in the subtree of $(s_0,\ldots,s_{k-1},u_k')$.

    Every infinite branch through $T$ is an infinite descending chain through
    $\mc L$. Thus, provably in $\mathsf{ACA}_0$, well-orderedness of $\mc L$
    implies the well-foundedness of $T$. In the other direction, for every
    infinite descending chain in $\mc L$ using arithmetical comprehension we find the infinite descending sequence $(),(s_0),(s_0,s_1),\ldots$ in $T$ that is defined to consist of sequences $(s_0,\ldots,s_{k-1})$ whose sub-trees contain infinitely many elements of the original descending sequence. 

    The linearity of this $\mc L$ is, of course, trivial to verify in $\mathsf{ACA}_0$. And since $\mathsf{PA}$ is precisely the set of first-order consequences of $\mathsf{ACA}_0$, the produced order $\mc L$ is $\mathsf{PA}$-verifiable.
\end{proof}

\subsection{Transfinite induction}
The statement that a linear order $\mc L$ is well-ordered is clearly second-order and
thus not expressible in $\mathsf{PA}$. However, as we will see, it is closely
related to the transfinite induction scheme for $\mc L$. 

For $\mc L$ a $\mathsf{PA}$-verifiable linear order, $\mathsf{TI}(\mc
L)$, the \emph{transfinite induction scheme for $\mc L$} 
consists of the formulas
\[(\forall x\in \mc L)((\forall y\prec^{\mc L} x)\varphi(y)\to \varphi(x))\to (\forall x\in \mc L)\; \varphi(x),\]
where $\varphi$ ranges over arbitrary arithmetical formulas that could contain other free variables.
\begin{lemma} 
    \label{TI_cons}
    Suppose $\mc L$ is a $\mathsf{PA}$-verifiable
    computable linear order. Then the theory $\mathsf{ACA}_0+\mathsf{WO}(\mc L)$ is a conservative extension of $\mathsf{PA}+\mathsf{TI}(\mc L)$.
\end{lemma}
\begin{proof} 
Obviously, $\mathsf{ACA}_0+\mathsf{WO}(\mc L)\supseteq
\mathsf{PA}+\mathsf{TI}(\mc L)$. Now let us show that every first-order
consequence of $\mathsf{ACA}_0+\mathsf{WO}(\mc L)$ is provable in
$\mathsf{PA}+\mathsf{TI}(\mc L)$. In order to do this we will show that any
model $\mathfrak{A}$ of $\mathsf{PA}+\mathsf{TI}(\mc L)$ can be extended to a
model of $\mathsf{ACA}_0+\mathsf{WO}(\mc L)$. Indeed, we extend $\mathfrak{A}$
to a model of the language of second-order arithmetic by all the subsets
$A\subseteq \mathfrak{A}$ that are definable in $\mathfrak{A}$ by an
arithmetical formula with parameters from $\mathfrak{A}$.

Let $X$ be in the
second-order part of $\mathfrak{A}$ such that $\phi(x)$ defines $X$. Then, since 
$\mathfrak{A}$ satisfies induction for $\phi(x)$, it also satisfies induction
for $X$ and since $X$ was arbitrary, it satisfies set induction. If
$\phi$ is an arithmetical formula with parameters $Y_1,\dots Y_n$ from
$\mathfrak{A}$, then let $\theta_i$ be the arithmetical first-order formulas
defining $Y_i$ and $\psi$ be the formula with every subformula of the
form $t\in
Y_i$ replaced by $\theta_i(t)$ for any term $t$. Then
\[ \mathfrak A\models \forall x (\psi(x) \lr \phi(x))\] 
and thus $\mathfrak A$ satisfies arithmetical comprehension for $\phi$.

It remains to show that $\mathfrak{A}$ satisfies $\mathsf{WO}(\mc L)$. As $\mathfrak{A}$ satisfies transfinite induction over $\mc L$ for arbitrary $\phi(x)$ with side parameters, it satisfies the second-order principles
\[\forall X\big ((\forall x\in \mc L)((\forall y\prec^{\mc L} x)y\in X\to x\in X)\to (\forall x\in \mc L)\; x\in X\big).\]
Which is clearly equivalent in $\mathsf{ACA}_0$ to the well-foundedness of $\mc L$:
\[\forall X \big(\exists x (x\in X\cap \mc L) \to (\exists x\in X\cap \mc L)(\forall y\in X\cap \mc L)
(x\preceq^{\mc L} y)\big).\]
Given that we assume $\mc L$ to be $\mathsf{PA}$-verifiable computable linear order, we conclude that we have $\mathfrak{A}\models \mathsf{ACA}_0+\mathsf{WO}(\mc L)$.
\end{proof}


\subsection{Uniform reflection and its iterations}\label{uniform_reflection_section}
For an arithmetical c.e.\ axiomatizable theory $T$ represented by a
$\Sigma_1$-formula $\mathsf{Ax}_T(x)$ expressing that $x$ is a G\"odel number of an axiom of $T$, we have the natural arithmetized provability predicate $\mathsf{Prv}_T(x)$. To be precise, the formula $\mathsf{Prv}_T(x)$ expresses that $x$ is a G\"odel number of an arithmetical sentence and there exists a proof in first-order logic of $x$ from some premises $y$ all of which satisfy $\mathsf{Ax}_T(y)$.

Full uniform reflection $\mathsf{RFN}(T)$ for a c.e.\ axiomatizable theory $T$ as
above  is the scheme
\begin{equation}\label{uniform_reflection}\forall x \big(\mathsf{Prv}_T(\gnmb{\varphi(\dot x)})\to \varphi(x)\big)\text{, where $\varphi(x)$ is any arithmetical formula.}\end{equation}
To simplify our notations we will consider $\mathsf{RFN}(T)$ to be the extension of $T$ by the schemata above. 

For a $\mathsf{PA}$-verifiable computable linear order $\mathcal{L}$ we define
the uniformly c.e.\ axiomatizable family of theories $(\mathsf{RFN}^{(\mc
L,a)}(T))_{a\in \mc L}$ as
\[\mathsf{RFN}^{(\mc L, a)}(T)=T+\bigcup\limits_{b\prec^{\mc L} a}\mathsf{RFN}(\mathsf{RFN}^{(\mc L, b)}(T))\]
Notice that $\mathsf{RFN}(\mathsf{RFN}^{(\mc L, b)}(T))$ as a set of theorems might depend on the choice of $\Sigma_1$-formula recognizing axioms of $\mathsf{RFN}^{(\mc L, b)}(T)$. Thus, in order to make the definition above formally correct, we, in fact, need to produce a family of formulas defining axioms of theories $\mathsf{RFN}^{(\mc L, a)}(T)$. 

We are going to define a $\Sigma_1$-formula $\mathsf{Ax}_T^{\mc L}(x,y)$ with
the intention that when we substitute a numeral $\underline{a}$ of $a$ instead
of $x$, then the resulting $\Sigma_1$-formula $\mathsf{Ax}_T^{\mc L}(\underline{a},y)$ should be the formula recognizing the axioms of $\mathsf{RFN}^{(\mc L, a)}(T)$. We construct $\mathsf{Ax}_T^{\mc L}(x,y)$ by G\"odel's Diagonal Lemma so that $\mathsf{Ax}_T^{\mc L}(x,y)$ is equivalent to the conjunction of the following conditions:
\begin{enumerate}
    \item $x\in\mc L$, 
    \item $y$ is an axiom of $T$ or for some $z\prec^{\mc L} x$, $y$ is an
        instance of the uniform reflection scheme $\mathsf{RFN}(\mathsf{RFN}^{(\mc L,z)}(T))$.
\end{enumerate}
Notice that the instance of $\mathsf{RFN}(\mathsf{RFN}^{(\mc L,z)}(T))$ formulated
as in \cref{uniform_reflection} employs $\mathsf{Prv}_{\mathsf{RFN}^{(\mc
L,z)}(T)}(x)$ as a subformula, which in turn is fully determined by the formula
$\mathsf{Ax}_{\mathsf{RFN}^{(\mc L,z)}(T)}(x)$. For the purpose of the fixed
point definition of $\mathsf{Ax}_T^{\mc L}(x,y)$, this instance of
$\mathsf{Ax}_{\mathsf{RFN}^{(\mc L,z)}(T)}(x)$ is 
$\mathsf{Ax}_{T}^{\mc L}(\underline z,x)$. That is, above we define
$\mathsf{Ax}_T^{\mc L}(x,y)$ in terms of its own G\"odel number. Since the
conjunction of the items (1) and (2) is clearly $\Sigma_1$, the standard proof
of the Diagonal Lemma will produce in this case a $\Sigma_1$-formula $\mathsf{Ax}_T^{\mc
L}(x,y)$.

Finally, we put 
\[\mathsf{RFN}^{\mc L}(T)=T+\bigcup\limits_{a\in \mc L} \mathsf{RFN}(\mathsf{RFN}^{(\mc L,a)}(T)).\]

\begin{lemma}\label{reflection_implies_TI}For any $\mathsf{PA}$-verifiable
computable order $\mc L$ and $a\in L$, the theory $\mathsf{RFN}^{(\mc L,a)}(\mathsf{PA})$ proves all instances of $\mathsf{TI}(\mc L\upharpoonright b)$, where $b\prec^{\mc L}a$.\end{lemma}
\begin{proof}
    Recall that L\"ob's Theorem \cite{lob1955solution} (in the case of $\mathsf{PA}$) states that if for some arithmetical sentence $\theta$ we have $\mathsf{PA}\vdash \mathsf{Prv}_{\mathsf{PA}}(\gnmb{\theta})\to\theta$, then $\mathsf{PA}\vdash \theta$. Note that L\"ob's theorem could, in fact, be considered just as a reformulation of the G\"odel's Second Incompleteness Theorem for the theory $\mathsf{PA}+\lnot \theta$.

    We will pick as the sentence $\theta$ the natural formalization of the
    statement of Lemma \ref{reflection_implies_TI} in the language of
    arithmetic. In the rest of the proof we reason in $\mathsf{PA}$ (i.e. by
    finitistic means) to prove that $\mathsf{Prv}_\mathsf{PA}(\theta)$ implies
    $\theta$, which by L\"ob's theorem will imply that $\theta$ is provable in
    $\mathsf{PA}$ and hence by soundness of $\mathsf{PA}$ we will conclude that
    $\theta$ is true which is precisely what we needed to show.
    
    We suppose that we have a proof $p$ of $\theta$, a $\mathsf{PA}$-verifiable computable order $\mc L$ and an element $a\in\mc L$. Our goal is produce a proof in $\mathsf{RFN}^{(\mc L,a)}(\mathsf{PA})$ of an instance 
     \begin{equation}\label{TI_inst}(\forall x\prec^{\mc L} b)((\forall y\prec^{\mc L} x)\varphi(y)\to \varphi(x))\to (\forall x\prec^{\mc L} b)\; \varphi(x),\end{equation}
     of  transfinite induction $\mathsf{TI}(\mc L\upharpoonright b)$.

     Note that the fact that $\mc L$ is a $\mathsf{PA}$-verifiable linear order implies that $\mathsf{PA}$ proves that $\mc L$ is a $\mathsf{PA}$-verifiable linear order. Therefore, since $\theta$ is provable in $\mathsf{PA}$, it is  provable in $\mathsf{PA}$ that,
     \[ \forall c\prec^{\mc L} b\,\mathsf{Prv}_{\mathsf{RFN}^{(\mc L,b)}(\mathsf{PA})}\big(\gnmb{(\forall x\prec^{\mc L} \dot c)((\forall y\prec^{\mc L} x)\varphi(y)\to \varphi(x))\to (\forall x\prec^{\mc L} \dot c)\; \varphi(x),}\big).\] 
     Henceforth, since $\mathsf{RFN}^{(\mc L,a)}(\mathsf{PA})$ contains $\mathsf{RFN}(\mathsf{RFN}^{(\mc L,b)}(\mathsf{PA}))$ we conclude that $\mathsf{RFN}^{(\mc L,a)}(\mathsf{PA})$ proves
     \begin{equation}\label{pre_TI_inst}\forall c\prec^{\mc L} b\big((\forall x\prec^{\mc L} c)((\forall y\prec^{\mc L} x)\varphi(y)\to \varphi(x))\to (\forall x\prec^{\mc L} c)\; \varphi(x)\big).\end{equation}
     Obviously (\ref{pre_TI_inst}) implies (\ref{TI_inst}) even over
     $\mathsf{PA}$, which concludes the proof of our claim and hence the lemma.
\end{proof}

\begin{theorem}[Feferman's Completeness Theorem] For every true arithmetical sentence $\varphi$ there exists a well-founded $\mathsf{PA}$-verifiable linear order $\mc L$ such that
  \[\mathsf{RFN}^{\mc L}(\mathsf{PA})\vdash \varphi.\]
\end{theorem}
\begin{proof}
    By Lemma \ref{arith-completeness} we could pick a $\mathsf{PA}$-verifiable computable well-order $\mc L_1$ such that \[\mathsf{ACA}_0\vdash \varphi\mathrel{\leftrightarrow} \mathsf{WO}(\mc L_1).\] By Lemma \ref{TI_cons},
    \[\mathsf{PA}+\mathsf{TI}(\mc L_1)\vdash \varphi.\]
    And for the order $\mc L$ that is obtained from $\mc L_1$ by adding one element $a$ on the very top, by Lemma \ref{reflection_implies_TI} we get
    \[\mathsf{RFN}(\mathsf{RFN}^{(\mc L,a)}(\mathsf{PA}))\vdash \varphi\]
    and hence
    \[\mathsf{RFN}^{\mc L}(\mathsf{PA})\vdash \varphi\]
\end{proof}

\section{Feferman's Completeness Theorem via Deduction Chains}\label{sec:history}

One might wonder whether Feferman could have used a different and technically
more transparent approach other than his intricate applications of the recursion
theorem. Indeed, the idea we are putting forward here is that two results available at the time could have led to such a proof. The first result, perhaps unsurprisingly, is L\"ob's theorem from 1955 \cite{lob1955solution} while the second is Sch\"utte's 1956 completeness theorem \cite[Satz 6]{schuette1956}
for $\omega$-logic, which showed that a canonical primitive recursive proof tree (indeed a Kalm\'ar-elementary tree) can be associated with every true arithmetic statement, entailing that $\omega$-logic with the primitive recursive $\omega$-rule (indeed Kalm\'ar-elementary $\omega$-rule) is already complete.\footnote{The completeness with the recursive  $\omega$-rule is usually credited to Shoenfield 
whose article \cite{shoenfield1959} appeared in 1959, while the canonical proof tree in many papers is associated with Mints' work (see \cite{mints1976}, \cite{sundholm1983}, \cite{franzen2004}). }   L\"ob's theorem, of course, was well known at the time, but even today it remains
largely unknown that Sch\"utte was the first mathematician who showed the primitive recursive completeness of $\omega$-logic\footnote{This was confirmed by G\"oran Sundholm, an expert on the $\omega$-rule, in conversations with the second author. It is unclear why this is the case. Even though Sch\"utte's paper was in German that should not have been an obstacle back then.} and introduced canonical search trees.

 \begin{definition} (The Tait calculus for $\omega$-logic)
 For classical logic, the so-called Tait calculus is very convenient (see
 \cite{tait1968} and \cite{schwichtenberg1977}). In it, one deduces finite
 sequences of formulas, called sequents. Sequents are referred to by capital
 Greek letters $\Delta,\Gamma,\Lambda,\Xi,\ldots$. The result of appending a
 formula $\theta$ to a sequent $\Gamma$ is written $\Gamma,\theta$. Formulas are
 assumed to be in negation normal form, i.e., negations appear only in front of
 atomic formulas, whereas negation of more complex formulas is a defined
 operation, using De Morgan's laws and dropping double negations. The other logical particles are $\wedge,\vee,\exists,\forall$. Literals are either atomic formulas or negated atomic formulas. Working in $\omega$-logic means that we can assume that formulas have no free variables. Below we are only concerned with $\omega$-logic based on the language of $\PA$. As a result, all literals are either true or false, and this can be checked by a primitive recursive algorithm. 
 The axioms of the infinitary calculus are those sequents that contain a true literal. The rules are the usual ones for deducing a conjunction, disjunction, or existential formula, plus the $\omega$-rule (for details see \cite{schwichtenberg1977}).
 \end{definition}
 \begin{definition}
 For any non-literal $\C$, a {\bf $\C$-deduction chain} is a finite string \[\Gamma_0,\Gamma_1,\ldots,\Gamma_k\]
of sequents $\Gamma_i$ constructed according to the following rules:
\begin{itemize}
\item[(i)] $\Gamma_0\; = \; \C$.
\item[(ii)] $\Gamma_i$ is not axiomatic for $i<k$.
\item[(iii)] If $i<k$, then $\Gamma_i$ is of the form
    \[\Gamma_i',\E,\Gamma_i''\]
    where $\E$ is not a literal and $\Gamma_i'$ contains only literals.
     $\E$ is said to be the {\em redex} of $\Gamma_i$.
\end{itemize}
Let $i<k$ and $\Gamma_i$ be reducible. $\Gamma_{i+1}$ is obtained from
 $\Gamma_i=\Gamma_i',\E,\Gamma_i''$
     as follows:
     \begin{enumerate}
    \item If $\E\equiv \E_0\vee \E_1$ then
      \[\Gamma_{i+1}\;=\; \Gamma_i',\E_0,\E_1,\Gamma_i''.\]
      \item  If $\E\equiv \E_0\wedge \E_1 $ then
      \[\Gamma_{i+1}\;=\; \Gamma_i',\E_j,\Gamma_i''\]
      where $j=0$ or $j=1$.
      \item  If $\E\equiv \exists x\,\F(x)$ then
      \[\Gamma_{i+1}\;=\; \Gamma_i',\F(\underline{m}),\Gamma_i'',\E\]
      where $m$ is the first number such that $\F(\underline{m})$ does not
occur in $\Gamma_0,\ldots,\Gamma_i$.
 \item If $\E\equiv \forall x\,\F(x)$ then
      \[\Gamma_{i+1}\;=\; \Gamma_i',\F(\underline{m}),\Gamma_i''\]
      for some $m$.
   \end{enumerate}

 The set of $\C$-deduction chains forms a tree, the {\em Stammbaum} ${\mathbb B}_{\C}$, labeled with
strings of sequents.\footnote{Later this tree was also called \emph{canonical tree}}
\end{definition}

 We will now consider two possible outcomes.
\begin{itemize}

\item {\em Case I:} ${\mathbb B}_{\C}$ is well-founded.
 \item[] Then ${\mathbb B}_{\C}$ yields an elementary recursive $\omega$-proof of $\C$. 
 So $\C$ is true. 
 
 \item {\em Case II:} ${\mathbb B}_{\C}$ is not well-founded.
 \item[] Then one shows that $\C$ is false by verifying that an infinite path through this tree contains only
 false formulas, proceeding by induction on the length of such formulas.
 
\end{itemize}

 Kleene's normal form theorem for arithmetic formulas (\cref{arith-completeness}) also assigns a linear computable ordering $\prec_\varphi$ to every sentence $\varphi$ such that the truth of $\varphi$ is equivalent to 
 $\prec_\varphi$ being a wellordering. This result is obtained by a series of syntactic transformations performed on $\varphi$, and it seems that it doesn't yield any new insights as to the truth of $\varphi$. Sch\"utte's completeness theorem, however, associates with $\varphi$ an ordering coming from
 a search tree for $\varphi$ in $\omega$-logic and it seems that this result is the more important one.
 
 \begin{definition} (Reflexive Induction)
 There is an immediate consequence of L\"ob's theorem, christened {\em reflexive induction}. It was singled out by Girard and Schmerl (see \cite{schmerl1979fine}, p. 337). Here we will use a slight variant 
 due to Beklemishev \cite{beklemishev2015}.

 A formula $\F(u)$ is said to be {\em $\prec$-reflexively progressive with respect to $T$} if \[T\vdash \forall x[\Prm{T}(\goed{\forall y\prec \dot{x}\,\F(y)})\to \F(x)].\]
\end{definition}

\begin{lemma} If $\F(u)$ is $\prec$-reflexively progressive with respect to $T$, then 
\[T\vdash \forall x\,\F(x).\]
\end{lemma}
\begin{proof} By L\"ob's theorem it suffices to show
\begin{eqnarray}\label{ref1} &&T\vdash \Prm{T}(\goed{\forall x\,\F(x)})\to \forall x\,\F(x).\end{eqnarray}
Reasoning in $T$, assume  $\Prm{T}(\goed{\forall x\,\F(x)})$. Then also  $
\forall x[\Prm{T}(\goed{\forall y\prec \dot{x}\,\F(y)})$, whence, by $\prec$-reflexive progressiveness,
$\forall x\,\F(x)$, ascertaining (\ref{ref1}).
\end{proof}

\kommentar{Assuming $T\vdash \forall x\,\Prm{T}(\goed{\forall y\prec \dot{x}\,\F(y)})$

$T\vdash \forall x[\forall y\prec x\,\Prm{T}(\goed{F(\dot{y})})\to F(x)]$,
we deduce \begin{eqnarray*} T\vdash \Prm{T}(\goed{\forall x\,F(x)}) &\to & \forall x\,\Prm{T}(\goed{\forall y \prec \dot{x}\,
F({y})}) \\ &\to & \forall x\,F(x)\end{eqnarray*}
L\"ob's Theorem for $T$ then yields 
\[T\vdash  \forall x\,F(x).\]
}

\begin{theorem} 
 Let $\C$ be an arithmetic sentence and $\mc L$ be the Kleene-Brouwer ordering on   $\mathbb{B}_{\C}$.
Then, for any node $\sigma\in \mathbb{B}_\psi$, \begin{eqnarray}\label{KB_R_C}\mathsf{RFN}(\mathsf{RFN}^{(\mc L,\sigma)}(\PA)) \vdash \Gamma_\sigma, \end{eqnarray}
where $\Gamma_\sigma$ is the sequent in the node $\sigma$. And in particular for the root node $\langle\rangle$:
\begin{eqnarray}\label{KB_R_G}\mathsf{RFN}(\mathsf{RFN}^{(\mc L,\langle\rangle)}(\PA)) \vdash \psi\end{eqnarray}
\end{theorem}
\begin{proof}We will prove the $\mc L$-reflexive progressivity of (\ref{KB_R_C}), which implies that (\ref{KB_R_C}) is provable in $\mathsf{PA}$ and hence true. For this, we reason in a $\PA$-formalizable manner.

Suppose that it is $\PA$-provable that for all $\tau\prec^{\mc L}\sigma$ we have
\begin{eqnarray}\label{MR1} \mathsf{RFN}(\mathsf{RFN}^{(\mc L,\tau)}(\PA))\vdash \Gamma_{\tau}.\end{eqnarray} 
Our goal is to prove (\ref{KB_R_C}).
The form of $\Gamma_\sigma$ can be effectively determined. The most interesting case is
the one where $\Gamma_\sigma$ is of the form \[\Gamma^{1}_\sigma,\forall x\theta(x),\Gamma^{2}_\sigma\] with redex $\forall x\theta(x)$. Then there are effectively determinable $\sigma_n\prec^{\mc L} \sigma$
such that $\Gamma_{\sigma_n}$ is of the form \[\Gamma^{1}_\sigma,\theta(\underline{n}),\Gamma^{2}_\sigma,\] and from (\ref{MR1}) it follows that it is $\mathsf{PA}$-provable that 
\begin{eqnarray}\label{MR2} && \forall n\;\big(\mathsf{RFN}(\mathsf{RFN}^{(\mc L,\sigma_x)}(\PA))\vdash \Gamma^{1}_\sigma,\theta(\underline{n}),\Gamma^{2}_\sigma\big).\end{eqnarray}
And since it is $\mathsf{PA}$-provable that all $\sigma_n$ are smaller than $\sigma$, we conclude that $\mathsf{PA}$-provably we have 
\begin{eqnarray}\label{MR3} && \forall n\;\big(\mathsf{RFN}^{(\mc L,\sigma)}(\PA)\vdash \Gamma^{1}_\sigma,\theta(\underline{n}),\Gamma^{2}_\sigma\big).\end{eqnarray}
Thus, applying reflection for $\mathsf{RFN}^{(\mc L,\sigma)}(\PA)$ to (\ref{MR3}) we conclude that we have \begin{eqnarray}\label{MR4} &&\mathsf{RFN}(\mathsf{RFN}^{(\mc L,\sigma)}(\PA))\vdash \forall y \big(\bigvee \Gamma^{1}_\sigma\lor \theta(y)\lor \bigvee\Gamma^{2}_\sigma\big),\end{eqnarray}
for a fresh new variable $y$.
Given that $y$ is not free in either $\Gamma^1_\sigma$ or in $\Gamma^2_\sigma$, we get 
\begin{eqnarray}\label{MR5} &&\mathsf{RFN}(\mathsf{RFN}^{(\mc L,\sigma)}(\PA))\vdash \Gamma_\sigma.\end{eqnarray} 
\end{proof}

\subsection{Upper bounds on the size of the Stammbaum}
One can show that the ordinal height of the Stammbaum ${\mathbb B}_{\C}$ of a true sentence $\C$
is less than $\omega^2$. More precisely, if the syntactic buildup of $C$ from literals takes $k$ steps,
then the height of ${\mathbb B}_{\C}$ is strictly less than $\omega\cdot (k+1)$. This follows by induction on $k$, by showing the more general result that any sequent containing a true formula of complexity $k$ has 
a Stammbaum of height strictly less than $\omega\cdot (k+1)$.

As a result of the foregoing, we then get that the Kleene-Brouwer ordering $\prec$ of ${\mathbb B}_{\C}$ 
has an order-type strictly less than $\omega^{\omega\cdot(k+1)}$, and thus
strictly less than $\omega^{\omega^2}$. Later in this paper we shall improve this bound by developing progressions of theories staying below $\omega^{\omega}$.

\section{A theorem of Ash and Knight in $\mathsf{ACA_0}$}\label{upperbounds_sect}
The goal of this section is to prove a version of a theorem of Ash and
Knight~\cite[Example 3]{ash1990} from computable structure theory in $\mathsf{ACA_0}$. In
its original formulation, this theorem connects membership in hyperarithmetical sets to the well-orderedness of certain linear orders.
We are now ready to state the version of Ash and Knight's theorem we have in mind. In the statement of this theorem and what follows $\eta$ denotes the order-type of the standard ordering on $\mathbb Q$. 
\begin{theorem}
\label{thm:akwo}
 Let $n\in\omega$. Then for every $\Pi_{2n+1}$ formula $\varphi(x)$ there is a uniformly computable sequence of linear orders $(\mc L_i)_{i\in\omega}$ such that $\mathsf{ACA}_0$ proves that for any $i$
 \[ \mc L_i \cong \begin{cases}
  \omega^n & \text{ if }\varphi(i)\\
  \omega^n(1+\eta) &\text{ if } \lnot\varphi(i).
\end{cases}\]

\end{theorem}
In order to prove \cref{thm:akwo} we will need a few more lemmas. The first is
the relativized version of Shoenfield's classical limit lemma. See~\cite{soare1999} for the classical proof of this result and more
background on computability theory.

Also to prove \cref{thm:akwo} we will need the notions of $\Delta^0_n(X)$ linear orders. Our
definitions follow Simpson~\cite{simpson2009} and can be done in
$\mathsf{ACA}_0$.

In order to define relative computability, fix for each $n\in \omega$ a
universal $\Pi^0_n$ formula $\pi_n(e,x,X)$ with free number variables $e$ and
$x$, and a free set variable $X$. Here $\Pi^0_n$ universality should be
$\mathsf{ACA}_0$-verifiable, i.e., for any $\Pi^0_n$ formula $\varphi(e,x,X)$ we should have
\[\mathsf{ACA}_0\vdash \forall e\exists e'\forall x \forall X (\varphi(e,x,X)\lr \pi_n(e',x,X)).\]

This allows us to define some standard notions within $\mathsf{ACA}_0$. We say that a set $Y$ is $\Pi^0_n(X)$ if there exists $e$ such that $x\in Y \lr \pi_n(e,x,X)$. We say that a set $Y$ is $\Sigma^0_n(X)$ if its complement $\mathbb{N}\setminus Y$ is $\Pi^0_n(X)$. We say that a set is $\Delta^0_n(X)$ if it is simultaneously $\Pi^0_n(X)$ and $\Sigma^0_n(X)$. A set $Y$ is $X$-computable if it is $\Delta^0_1(X)$. A set $Y$ is \emph{computable} if it is $\emptyset$-computable.

A \emph{linear order $\mc L=(L,\peql)$} is $\Delta^0_n(X)$ if $\peql$ is
$\Delta^0_n(X)$. Notice that this implies that the
universe $L$ is also $\Delta^0_n(X)$, since by reflexivity $x\in L$ iff
$x\peql x$. If $\mc L$ is $\Delta^0_1(X)$, then we say that $\mc L$ is
$X$-computable. We will denote linear orders by calligraphic letters and
their universes $L$ by the corresponding capital letters.
A sequence $(\mc A_i)_{i\in\omega}$ of linear orders is \emph{uniformly
$\Delta^0_n(X)$} if there is a $\Delta^0_n(X)$ set $S$ such that its $i^\text{th}$
column is equal to $\peqai$, i.e.,
\[ S^{[i]}=\{x: \langle i,x\rangle\in S\}=\peqai. \]
As usual, a sequence of linear orders is \emph{uniformly $X$-computable} if it
is uniformly $\Delta^0_1(X)$.

In a standard manner, using Ackermann membership, we encode finite sets of
naturals by natural numbers, i.e., the elements of a number $A$ with the binary expansion $a_{n-1}\ldots a_0$ are all $i<n$ such that $a_i=1$. Let $(A_i)_{i\in\omega}$ be a sequence of finite sets of naturals. We say that a set $X$ is the limit of $(A_i)_{i\in\omega}$ and write $X=\lim A_i$ if
\[\begin{array}{lcl}
    x\in X & \text{iff } & \exists i (\forall j>i) x\in A_j,\\
    x\not\in X & \text{iff } & \exists i(\forall j>i) x\not\in A_j.
  \end{array} \] 

  \begin{lemma}[$\mathsf{ACA}_0$; Limit Lemma]\label{lem:limitcomp}
  Fix a set $X$. A set $Y$ is
  $\Delta^0_2(X)$ if and only if there is an $X$-computable sequence of finite sets $(A_i)_{i\in\omega}$ such that $Y=\lim A_i$.
\end{lemma}
\begin{proof}
  $(\Ra)$.
  As $Y$ is $\Delta^0_2(X)$ there are $e_0$, $e_1$ such that
  \begin{align*}
      y\in Y &\lr \pi_2(e_0,y,X)\lr \forall u\exists v R(e_0, y,u,v,X)\\
      y\not\in Y &\lr \pi_2(e_1,y,X)\lr \forall u\exists v R(e_1,y,u,v,X)
  \end{align*}
  for some formula $R$ containing only bounded quantifiers. First, for fixed $i$, define sets $B_i$ and $C_i$ as follows.
  \begin{align*}
    y\in B_i &\lr y<i \land (\exists u<i) (\forall v<i) R(e_0,y,u,v,X)\\
    y\in C_i &\lr y<i \land (\exists u<i) (\forall v<i) R(e_1,y,u,v,X)
  \end{align*}
  The sequences $(B_i)_{i\in\omega}$ and $(C_i)_{i\in\omega}$ are clearly $X$-computable and it is easy to see that
  \begin{align*}
    y\in Y &\lr \exists i (\forall j>i) y\in B_i,\\
    y\not\in Y &\lr \exists i (\forall j>i) y\in C_i.
  \end{align*}
  Let $age_B(i,y)=\min (\{i+1\}\cup\{k\le i: y\not\in B_{i-k}\})$ and
  $age_C(i,y)=\min
  (\{i+1\}\cup\{k\le i\mid y\not\in C_{i-k}\})$. Clearly both $age_B$ and $age_C$ are $X$-computable functions from $\mathbb{N}^2$ to $\mathbb{N}$.
  We define 
  \[ A_i=\{y<i: age_B(i,y)>age_C(i,y)\}.\]
  The sequence $(A_i)_{i\in\omega}$ is clearly uniformly $X$-computable. To see
  that $Y$ is its limit notice that if $y\in Y$, then there is $i_0$ such that
  $(\forall j>i_0)y\in B_j$ and that for every $i$ with $y\in C_i$, there is
  $j>i$ such that $y\not\in C_j$. Let $i_1$ be the least such $j$ with $y\not\in
  C_j$ greater than $i_0$. Then for all $j>i_1$, $age_B(j,y)>age_C(j,y)$ and
  hence $y\in A_j$. That $y\not\in Y$ if there is $i$ such that for all $j>i$ $y\not\in A_j$ can be shown in a similar manner.

  $(\La)$. Assume $(A_i)_{i\in\omega}$ is uniformly $X$-computable and let $Y$ be the limit of $(A_i)_{i\in\omega}$. There are $e_0$ and $e_1$ such that
  \[ x\in A_i \lr \pi_1(e_0, \langle i,x\rangle) \quad \text{and} \quad x\not\in A_i \lr \pi_1(e_1,\langle i,x\rangle).\]
  From this, one can easily extract a pair of $\Pi^0_2$ formulas which, given parameter $X$, define the membership relation of $Y$, respectively, its negation. Thus $Y$ is $\Delta^0_2(X)$.
\end{proof}
From \cref{lem:limitcomp} we will now derive a limit lemma for linear orders. 
\begin{definition}
    A sequence $(\mc A_i)_{i\in\omega}$ of finite linear orders is
\emph{locally coherent} if $\preceq^{\mc A_i}\upharpoonright (A_i\cap A_{i+1})
=\preceq^{\mc A_{i+1}}\upharpoonright (A_i\cap A_{i+1})$, for each $i$.
\end{definition}
 We will be interested in the situation when a linear order 
$\mc L$ is the limit of a locally coherent sequence of finite linear orders
$(\mc A_i)_{i\in\omega}$. In this case, we think about $\mc L$ as the
order obtained by the following procedure. We start with the order $\mc A_0$.
At each stage $i$ we switch from $\mc A_{i}$ to $\mc A_{i+1}$ by removing all
the elements $x\in A_i\setminus A_{i+1}$ and adding all the elements in $x\in
A_{i+1}\setminus A_i$. The order $\mc L$ consists of all the elements $x$ that
are present in some $\mc A_i$ and aren't removed in any $\mc A_j$, for $j>i$.
Whether $x\peql y$ can be checked
in any $\mc A_{i}$ such that both $x,y$ aren't removed at any stage $\ge i$.

We say that $S$ is a partial successor relation on a linear order $\mc L$ if
\begin{enumerate}
\item $\langle x,y\rangle \in S\Rightarrow x\prec^{\mc L}y$,
\item $\langle x,y\rangle \in S\Rightarrow (\forall z\in L) (z\peql x\lor
  y\peql z)$.
\end{enumerate}
A finite linear order with a partial successor relation $\mc A$ is a triple $(A,
\preceq^{\mc A},S^{\mc A})$, where  $\preceq^{\mc A}$ is a finite linear order
and $S^{\mc A}$ is a finite set of pairs that is a partial successor relation
on $\preceq^{\mc A}$. We say that a sequence of finite linear orders with
partial successor relations $(\mc A_i)_{i\in \omega}$ is \emph{locally
coherent} if
$\preceq^{\mc A_i}\upharpoonright (A_i\cap A_{i+1}) =\preceq^{\mc
A_{i+1}}\upharpoonright (A_i\cap A_{i+1})$ and $S_i\upharpoonright (A_i\cap
A_{i+1}) \subseteq S_{i+1}\upharpoonright (A_i\cap A_{i+1})$. As for sequences of
finite linear orders we will be interested in the situation when a linear order
with the total successor relation $\mc L=\langle \preceq^{\mc L},S^{\mc
L}\rangle$ is the limit of a locally coherent sequence of finite linear orders
with partial successor relations $(\mc A_i)_{i\in \omega}$.

Note that for a linear order without successor relation $\mc L=\lim \mc A_i$ and
successive $x,y\in \mc L$, new elements $z$ might appear and disappear between
$x$ and $y$ inside $\mc A_i$'s for indefinitely large $i$. However, for a linear order with successor relation $\mc L=\lim \mc A_i$ and neighboring elements $x,y\in \mc L$, for large enough stages (after the point when both $x$ and $y$ will no longer be removed) we can guarantee that no more new elements will appear in $\mc A_i$'s between $x$ and $y$. 

Using this definition we get the following corollary of~\cref{lem:limitcomp}.
\begin{corollary}[$\mathsf{ACA}_0$]\label{cor:limloccoherent}
  Fix a set $X$. A linear order $\mc L$ is $\Delta^0_2(X)$
  if and only if there is an $X$-computable locally coherent sequence $(\mc A_i)_{i\in\omega}$ of
  finite linear orders with limit $\mc
  L$, i.e., 
  \begin{align*}x\in L &\lr \exists i(\forall j>i)\ x\in A_j\\ x\not\in L &\lr \exists i (\forall j>i)\ x\not\in A_j\\ x\peql y &\lr \exists
  i (\forall j>i)\ x\preceq^{\A_j} y.\end{align*}
\end{corollary}
\begin{proof}
  $(\Ra)$. From \cref{lem:limitcomp} we get a sequence $(B_i)_{i\in\omega}$ of
  finite sets with limit $\peql$. We may assume that no $B_i$ contains elements
  greater than $i$. We define a locally coherent sequence
  $(\A_i)_{i\in\omega}$ as follows. Let $\A_0$ be the empty order. Assume we have defined $\A_i$. To define $\A_{i+1}$ first
  let $C_{i+1}$ be the largest subset of $\{ x: \langle x,x\rangle\in B_{i+1}\}$ linearly ordered by $B_{i+1}$. 
  Then let 
  \[ A_{i+1}= C_{i+1} \text{ and } \preceq^{\A_{i+1}}=B_{i+1}\restrict A_{i+1}.\]

  This sequence is clearly uniformly $X$-computable. To see that it has limit
  $\mc L$, first consider $x,y\in L$ such that $x\peql y$. There exists $i$
  such that for all $j>i$ and $z,w\le \max(x,y)$ we have $\langle z,w\rangle\in B_j\iff z\peql w$. By definition, we then have
  that $x\preceq^{\A_{j}} y$. On the other hand, suppose that $x\not\in L$, then there is $i$ such that for all $j>i$, $\langle x,x\rangle\not\in B_i$. So for all $j>i$, $x\not\in A_j$.

  $(\La)$. A locally coherent sequence of finite linear orders $(\mc A_i)_{i\in
  \omega}$ is a special case of a sequence of finite sets. Thus this implication
  is a direct consequence of \cref{lem:limitcomp}.
%
\end{proof}
Using the same ideas as in its proof, we can extend \cref{cor:limloccoherent} to
work for locally coherent sequences of finite linear orders with partial successor
relations.

\begin{corollary}[$\mathsf{ACA}_0$]\label{cor:limloccoherentsucc}
  Fix a set $X$. A linear order
  $\mc L$ with successor relation $S^\mc L$ is $\Delta^0_2(X)$
  if and only if there is an $X$-computable locally coherent sequence $(\mc
  A_i)_{i\in\omega}$ of
  finite linear orders with successor relation with limit $\mc L$.
\end{corollary}
We are now ready to prove two lemmas crucial for our proof of \cref{thm:akwo}.
They give a sufficient condition for the existence of
computable linear orderings of order-type $\omega\cdot \mc L$ for any linear
ordering $\mc L$. Variations of these lemmas were first proven by Fellner~\cite{fellner1977}.
\begin{lemma}[$\mathsf{ACA_0}$]\label{lem:jumpinv1}
  Fix $X$ and let $\mc L$ be a $\Delta^0_2(X)$ linear order with a least element.  Then there is an $X$-computable linear order of order-type $\omega\cdot \mc L$ with an $X$-computable successor relation and $0$ as its least element.\footnote{We would like to thank Patrick Lutz and the anonymous referee who independently spotted a gap in the proof of Lemma \ref{lem:jumpinv1} from an earlier version of the paper.}
\end{lemma}
\begin{proof}
Assume without loss of generality that $0$ is the least element in $\mc L$. As
  $\mc L$ is $\Delta^0_2(X)$, by \cref{cor:limloccoherent} there is
  a uniformly $X$-computable locally coherent sequence $(\A_i)_{i\in\omega}$ of linear orders 
  with limit $\mc L$. We may furthermore assume that $\A_i$ does not contain elements greater than $i$ and that each $\A_i$ has the correct information about $0$.

We define $\mc L'$  of order-type $\omega\cdot \mc L$ as the union
  $\bigcup_{s\in\mathbb{N}} \mc L_s'$ of an $X$-computable sequence of
  finite linear orders $\mc L_0'\subseteq \mc L_1'\subseteq \ldots$ where $\mc L_i'$ has partial successor relation $S_i$. Together
  with $\mc L_s'$ we define an $X$-computable sequence of order-preserving surjections
  $f_s\colon \mc L_s'\to \mc A_s$ with the goal that $f=\lim f_s: \mc L' \to \mc L$ is a surjective order-preserving function labeling the ``$\omega$-blocks'' of $\mc L'$ with elements of $\mc L$.

We put $\mc L_0'= \mc A_0$ and let $f_0$ be the identity function. We will now define $\mc L_{s+1}'$ and $S_{s+1}$ as extensions of $\mc L_s'$ and $S_s$ respectively and $f_{s+1}$ given $f_s$ as follows. During the construction, we will add fresh elements to $\mc L_{s+1}'$ where a natural number $n$ is a \emph{fresh element} if $n>s+1$ and has not been used in the construction before.
\begin{enumerate}
  \item If there is a least natural number $a_0\in A_{s}$ such that the subordering of $\mc A_s$ consisting of natural numbers less than or equal to $a_0$ is not a subordering of $\mc A_{s+1}$, then let $\mc A'$ be the suborder of $\mc A_{s+1}$ 
  consisting of natural numbers less than $a_0$ and go to (2). Otherwise, let $\mc A'=\mc A_{s}$ and go directly to (3).
  \item For all elements $b\in A_{s}$ such that $b\geq a$ and all $x\in L_{s}'$ with $f_s(x)=b$, let $f_{s+1}(x)$ be the next element to the left of $b$ in $\mc A'$. If there are successive elements $x,y\in L_{s}'$ with $f_{s+1}(x)=f_{s+1}(y)$ for which $S_{s}$ is undefined, set $S_{s+1}(x,y)$.
  \item For each element $a\in A_{s+1}\setminus A'$ we add to $\mc L_{s+1}'$ a fresh element $x_a$ such that $f_{s+1}(x_a)=a$. We compare $x_a$ with other elements $y\in \mc L_{s+1}'$ as follows: $x_a\preceq^{\mc L'_{s+1}}y$ iff $a\preceq^{A_{s+1}}f_{s+1}(y)$.
  \item At last, grow every labeled block by an element to the right. I.e., for every $a \in L_{s+1}'$ such that for all $b\succeq^{\mc L_{s+1}'} a$, $f_{s+1}(b)\neq f_{s+1}(a)$, add a fresh element $x_a$, set $f_{s+1}(x_a)=f_{s+1}(a)$, $S_{s+1}(a,x_a)$ and define $\preceq^{\mc L_{s+1}'}$ accordingly.
\end{enumerate}
This finishes the constructions of $\mc L_{s+1}'$, $S_{s+1}$ and $f_{s+1}$. Let $\mc L'=\bigcup_{s\in\omega} \mc L_s'$, $f=\lim f_{s}$ and $S=\bigcup_{s\in\omega} S_s$.

\emph{Verification:}
First, note that $\mc L'$ is computable. For any two elements $x,y\in L'$ with $m=\max\{x,y\}$, we have by construction that\[  x\preceq^{\mc L'}y \iff x\preceq^{\mc L'_m} y.\] For the successor relation, the same reasoning applies except that at step (2) of the construction, we define the relation between elements added at previous stages. However, note that if at stage $s+1$, $S_{s+1}(x,y)$ is defined by step (2), then the element $x$ must be the rightmost element of its "labeled block" and thus it was added by step (4) at stage $s$. Hence, for all $x,y\in L'$ with $m=\max\{x,y\}$ \[S(x,y)\iff S_{m+1}(x,y).\] 
Next, note that $f$ is well-defined. Whenever $f_{s+1}(x)\neq f_{s}(x)$, by construction $f_{s+1}(x)<f_{s}(x)$ and hence $f_s(x)$ stabilizes in the limit for any $x\in L'$. To see that $range(f)=L$, first suppose that $a\in L$. By properties of $(\mc A_i)_{i\in\omega}$, there is a least $i$ such that for all $j>i$, $\mc A_i\restrict a=\mc A_j\restrict a$. By construction $f_{i+1}(x)=a$ for some $x\in L_{i+1}$ and $f_{j}(x)=a$ for all $j>i+1$. On the other hand, suppose that $a\in range(f)\setminus L$. Then there is $i$ such that for all $j>i$, $\mc A_i\restrict a=\mc A_j\restrict a$. Thus, for all $j>i$, $a\not\in range(f_{j})$, a contradiction. That $f$ is order-preserving follows trivially from the construction. It remains to show that the interval of elements labeled with $a\in L$ is of order-type $\omega$. To see this let $i$ be such that there is $x\in L_i'$ with $f_j(x)=a$ for all $j>i$. After stage $i$, the interval of elements with $f(x)=a$ will grow to the right by at least one and at most finitely many elements at each stage, but no elements will be added to the left or in between elements in the interval. Thus the order-type of each labeled interval is $\omega$. As $f$ is order-preserving and onto $\mc L$, it follows that $\mc L'$ has order-type $\omega\cdot \mc L$ as required.

\end{proof}
\begin{lemma}[$\mathsf{ACA_0}$]\label{lem:jumpinv2}
  Fix~$X$~and let $\mc L$ be a $\Delta^0_2(X)$ linear order of order-type $\omega\cdot\mc K$ for some linear order $\mc K$. Further, assume that $\mc L$ has $0$ as its least element and that the successor relation in $\mc L$ is $\Delta^0_2(X)$. Then there is an $X$-computable linear order isomorphic to $\mc L$ with $0$ as its least element.
\end{lemma}
\begin{proof}
  As $\mc L$ is $\Delta^0_2(X)$, by \cref{cor:limloccoherentsucc} there is
  a uniformly $X$-computable locally coherent sequence of finite linear orders
  $(\A_i)_{i\in\omega}$ with partial successor relation and $\mc L$ as its
  limit. Furthermore, we can assume without loss of generality that $0$ is the least element of all $\A_i$.
  
  We build an $X$-computable linear order $\mc L'$ isomorphic to $\mc L$ as the
  union of a sequence of extending finite orders $\mc L_s'$. As in the proof of
  \cref{lem:jumpinv1} we also define order-preserving surjections $f_s: \mc L'_s\ra \A_s$.

  We put $\mc L_0'=\mc A_0$ and let
  $f_0$ be the identity function. Assume we have defined $\mc L_s'$ and $f_s$.
  For elements $x \in \mc L_{s}'$ we put
  \[f_{s+1}(x)=\max_{\preceq^{A_{s}}}\{a\in A_{s+1}\cap A_s\mid a\preceq^{\mc A_{s}}f_s(x)\}.\]
  The additional elements of $\mc L_{s+1}'$ are $x_a$ for $a\in \mc A_{s+1}$ such that there are no $y\in \mc L_s'$ with $f_{s+1}(y)=a$. We put $f_{s+1}(x_a)=a$ and insert $x_a$ in $\mc L_{s+1}'$ above all $y$ with $f_{s+1}(y)\prec^{\mc A_{s+1}}a$ and below all $y$ with $f_{s+1}(y)\succ^{\mc A_{s+1}}a$.
  
  This finishes the construction. Let $\mc L'=\lim_s \mc L_s'$.

  \emph{Verification:} Clearly $\mc L'$ is $X$-computable. It remains to show
that $\mc L'\cong \mc L$. First of all notice that the order $\mc L'$ splits
into the blocks $B_a$, for $a\in \mc L$, where $B_a=\bigcup\limits_{s\ge
s_a}f_s^{-1}(a)$ and $s_a$ is the first stage such that
$a\in\bigcap\limits_{s\ge s_a} A_s$. In order to finish the proof it is enough
to show that all $B_a$ are finite and non-empty. For a given $a\in \mc L$
consider its successor $b$. In all the orders $\mc A_{s}$, for $s\ge
\max(s_a,s_b)$, the element $b$ is the successor of $a$ according to $S^{\mc A_s}$. Thus no elements will appear in between $a$ and $b$ in the orders $\mc A_s$. Therefore, no new elements will be added to $B_a$ after the stage $\max(s_a,s_b)$, and hence $B_a$ is finite. The block $B_a$ is non-empty since by construction $f^{-1}_{s_a}(a)=\{x_a\}$.\end{proof}

Combining \cref{lem:jumpinv1,lem:jumpinv2} we get a formal version of a special
case of a theorem
due to Ash~\cite{ash1991} as stated in~\cite[Theorem 9.11]{ash2000}: If a linear
order $\mc L$ has an isomorphic copy that is $\Delta^0_3$, then $\omega\cdot \mc L$ has a
computable copy. The proof of this in two steps, first applying
\cref{lem:jumpinv1} with $X=\emptyset'$ and then using \cref{lem:jumpinv2}, was
communicated to us by Andrey Frolov.

\begin{proof}[Proof of \cref{thm:akwo}]
  If $X$ is $\Pi_{2n+1}$, then, for some $\Sigma_{2n}$ set $Y$
  \[ x\in X \lr \forall y \langle x,y\rangle \in Y.\]
  We can now build a uniformly $\Delta_{2n}$ sequence of linear orders $(\mc C_i)_{i\in\omega}$ such that if $x\in X$, then $\mc C_i$ is the linear order only containing the element $0$ and if $x\not \in X$, then $\mc C_i$ has order-type $1+\eta$ where the first element is $0$. The construction is again in stages where $\mc C_{i,s}$ contains only one element $0$ if for all $y<s$ $\langle x,y\rangle\in Y$ and otherwise we build a copy of $\eta$ after the $0$.
  The structure $\mc C_{i}=\lim_s \mc C_{i,s}$ clearly is as required.

  Having defined $(\mc C_i)_{i\in\omega}$ we apply \cref{lem:jumpinv1} and then to the resulting sequence \cref{lem:jumpinv2} and so on, until, after $n$ repetitions of this process we get a computable sequence of linear orders $\mc L_i$ such that
  \[ \mc L_i \cong \begin{cases}
   \omega^n & \text{ if }i\in S\\
   \omega^n(1+\eta) &\text{ if } i\not\in S.
 \end{cases}\]
 Of course, the way \cref{lem:jumpinv1} and \cref{lem:jumpinv2} are stated, we can not do this right away. We require more uniformity. However, looking at the proof of \cref{lem:jumpinv1} the only step that is not uniform is the assumption that the first element is $0$. However, in our case, this is no obstacle, as we have defined $(\mc C_i)_{i\in\omega}$ so that all $\mc C_i$ have $0$ as its first element.
\end{proof}

\begin{corollary}\label{sen2wo} For any $\Pi_{2n+1}$ sentence $\varphi$ there exists a computable linear order $\mc L$ such that $\mathsf{ACA}_0$ verifies the following: 
  \begin{enumerate}
     \item  $\mc L$ is a linear order;
  \item $\mc L\cong \omega^n$ if $\varphi$;
  \item $\mc L\cong \omega^n(1+\eta)$, if $\lnot \varphi$;
  \item $\varphi\mathrel{\leftrightarrow}\mathsf{WO}(\mc L)$.
  \end{enumerate}
\end{corollary}
\begin{proof} We consider $\varphi$ as $\varphi(x)$, then apply \cref{thm:akwo} to it and finally put $\mc L=\mc L_0$.
\end{proof}

%
Combining Lemma \ref{TI_cons} and Corollary \ref{sen2wo} we get
\begin{corollary} \label{TI_Scheme_theory}For any $\Pi_{2n+1}$ sentence $\varphi$ there exists a computable linear order $\mc L$ such that
  \begin{enumerate}
  \item $\mathsf{PA}$ proves that $\mc L$ is a linear order;
  \item $\mc L\cong \omega^n$ if $\varphi$ is true;
  \item $\mc L\cong \omega^n(1+\eta)$, if $\varphi$ is false;
  \item $\mathsf{PA}+\varphi\equiv \mathsf{PA}+\mathsf{TI}(\mc L)$.
  \end{enumerate}
\end{corollary}

Now we are ready to prove Theorem \ref{upperbound}

\upperbound*
\begin{proof}
    We use Corollary \ref{TI_Scheme_theory} to obtain computable $\mathsf{PA}$-verifiable linear order $\mc L_1\simeq \omega^n$ such that $\mathsf{PA}+\mathsf{TI}(\mc L_1)\vdash \varphi$. Next, we consider the order $\mc L$ obtained from $\mc L_1$ by adding a new greatest element, i.e., $\mc L\simeq \omega^n +1$. By Lemma \ref{reflection_implies_TI} we see that $\mathsf{RFN}^{\mc L}(\mathsf{PA})$ proves all instances of $\mathsf{TI}(\mc L_1)$ and thus $\varphi$, which concludes the proof of Theorem \ref{upperbound}.
\end{proof}

\section{Lower bound for Feferman's completeness}\label{lowerbounds_sect}

In order to provide lower bounds for Feferman's completeness theorem we will use
the fact that for any  computable linear ordering $\mc
L$ and $n\ge 1$ there is a $\Sigma_{2n}$ formula $\mc L{\upharpoonright}_x
{\hookrightarrow}_n y$ such that for any $a\in \mc L$ and $\alpha<\omega^n$
(represented as a Cantor normal form term) the formula $\mc L{\upharpoonright}_x
{\hookrightarrow}_n \alpha$ expresses that the initial segment of $\mc L$ up to
$\alpha$ embedds into $\alpha$. We will furthermore need the following property being $\mathsf{PA}$-verifiable for $\mathsf{PA}$-verifiable linear orders $\mc L$:
\begin{equation}\label{emb_property}
(\forall \alpha<\omega^n)(\forall a\in \mc L)((\mc L{\upharpoonright}_a {\hookrightarrow}_n \alpha)\mathrel{\leftrightarrow}(\forall b\prec^{\mc L} a)(\exists \beta<\alpha)  (\mc L{\upharpoonright}_b {\hookrightarrow}_n \alpha))    
\end{equation}

We will do this by defining for $n\ge 0$  more general $\Sigma_{2n}$ formulas
$[x,y)^{\mc L} {\hookrightarrow}_n z$ such that for $a\preceq^{\mc L} b$ and
$\alpha<\omega^n$ the formula $\mc L{\upharpoonright} [a,b) {\hookrightarrow}_n
\alpha$ expresses that the interval $[a,b)^{\mc L}$ in $\mc L$ embeds into $\alpha$. Then we put $\mc L{\upharpoonright}_x {\hookrightarrow}_n y$  to be \[(\exists z\in \mc L)((\forall w\in\mc L)(z\preceq^{\mc L} w)\land ([z,x){\hookrightarrow}_n y)).\]

We define this formulas by induction on $n$. For the case of $n=0$, the only
eligible ordinal $\alpha$ is $0$ and thus we simply put $[x,y)^{\mc L}
{\hookrightarrow}_n z$ to be $x=y$. For $n+1$, given an ordinal $\alpha=
\omega^{m_1}+\ldots+\omega^{m_k}$, where all $m_i\le n$, we express the property $[a,b)^{\mc L}{\hookrightarrow}_{n+1} \alpha$ as
\[(\exists c_0\preceq^{\mc L}\ldots\preceq^{\mc L} c_k)\big(c_0=a\land c_k=b\land (\forall i<k)(\forall d\in [c_i,c_{i+1})^{\mc L})(\exists \beta<\omega^{m_i})( [c_i,d){\hookrightarrow}_n \beta)\big),\]
which is a $\Sigma_{2n+2}$ formula. In a standard manner we go from the above formulas for individual $\alpha$ to a single $\Sigma_{2n+2}$ formula $[x,y)^{\mc L} {\hookrightarrow}_n z$. Verification of (\ref{emb_property}) is also routine.

Let $\mc C$ be the standard order of order-type $\varepsilon_0$, whose elements are nested Cantor normal form terms. For $\alpha<\varepsilon_0$ we put $\mathsf{RFN}^{\alpha}(T)$ to be $\mathsf{RFN}^{(\mc C,\alpha)}(T)$. 

\begin{lemma} \label{refNF} For any $n\ge 1$ and $\mathsf{PA}$-verifiable computable linear order $\mc L$ it is $\mathsf{PA}$-provable that for any $a\in \mc L$ and $\alpha<\omega^n$ we have
  \begin{equation}\label{refNF1}\mathsf{RFN}(\mathsf{RFN}^{\alpha}(T))+(\mc L{\upharpoonright}_{a} {\hookrightarrow}_n \alpha)\supset \mathsf{RFN}(\mathsf{RFN}^{\mc L,a}(T)).\end{equation}
\end{lemma}
\begin{proof} We fix $n$ and prove this using L\"ob's theorem. We reason in finitistic manner (which here simply means $\mathsf{PA}$-formalizable) and assume that it is $\mathsf{PA}$-provable that for all $\mathsf{PA}$-verifiable computable linear orders $\mc L$, ordinals $\alpha<\omega^n$ and $a\in \mc L$ we have (\ref{refNF1}). We consider some $\mc L$, $\alpha<\omega^n$ and $a\in \mc L$ and claim that (\ref{refNF1}) holds. 

Note that for any c.e.\ $U$ extending $\mathsf{PA}$, we have
$\mathsf{RFN}(U)+\varphi\supseteq \mathsf{RFN}(U+\varphi)$. And that
$\mathsf{RFN}(U)+(U\supseteq V)\supseteq \mathsf{RFN}(V)$. Thus in order to
prove (\ref{refNF1}) it is enough to show that $\mathsf{PA}+(\mc
L{\upharpoonright}_{ a} {\hookrightarrow}_n \alpha) $ proves that any finite
fragment $U$ of $\mathsf{RFN}^{\mc L,a}(T)$ is contained in
$\mathsf{RFN}^{\alpha}(T)+\varphi$ for some true $\Sigma_{2n}$-sentence
$\varphi$ \footnote{We take $\Sigma_{2n}$ here since it is what we need for the
proof. Any other class $\Sigma_{k}$/$\Pi_{k}$ would be eligible, but not
the class of all arithmetical formulas, since we cannot finitistically talk
about the truth of arbitrary arithmetical sentences.}. 

We further reason in $\mathsf{PA}+(\mc L{\upharpoonright}_{a}
{\hookrightarrow}_n \alpha)$ and consider a finite fragment $U$ of
$\mathsf{RFN}^{\mc L,a}(T)$. Notice that $U$ is a subtheory of some
$\mathsf{RFN}(\mathsf{RFN}^{\mc L,b}(T))$, where $b\prec^{\mc L}a$. By
(\ref{emb_property}) we can deduce from $\mc L{\upharpoonright}_{a} {\hookrightarrow}_n \alpha$ that there is $\beta<\alpha$ such that $\mc L{\upharpoonright}_{b} {\hookrightarrow}_n \beta$. We take $\mc L{\upharpoonright}_{b} {\hookrightarrow}_n \beta$ as our $\varphi$. By the premise that (\ref{refNF1}) is already $\mathsf{PA}$-provable we see that $U$ is contained in $\mathsf{RFN}(\mathsf{RFN}^{\beta}(T))+(\mc L{\upharpoonright}_{b} {\hookrightarrow}_n \beta)$ and hence $U$ is contained in $\mathsf{RFN}^{\alpha}(T)+(\mc L{\upharpoonright}_{b} {\hookrightarrow}_n \beta)$.\end{proof}

\lowerbound*
\begin{proof} Indeed by Lemma \ref{refNF} the theory
    \[T=\mathsf{RFN}^\alpha(\mathsf{PA})+\text{``all true
    $\Sigma_{2n}$-sentences''}\] contains $\mathsf{RFN}^{\mc L}(\mathsf{PA})$,
    for any $\mathsf{PA}$-verifiable computable well-ordering $\mc L$ of
     order-type $\le \omega^n$. However, the arithmetically sound
    theory $T$ has a $\Sigma_{2n}$-set of theorems and thus there is a true
    $\Pi_{2n}$ sentence $\varphi$ that is not provable in this theory. Hence
    $\varphi$ can not be proved by any $\mathsf{RFN}^{\mc L}(\mathsf{PA})$ of the considered form. 
\end{proof}

\section{Reflection iterated along elements of $\mc O$} \label{Kleene_O_Section}
Feferman's completeness theorem was based on progressions along notations in
Kleene's $\mathcal O$, but in \cref{upperbound} we iterate reflection along
a computable well-order. There is a subtle difference in that notations in Kleene's
$\mathcal{O}$ yield linear orderings with properties that are generally
not possessed by computable linear orderings. In particular, the successor
relation, the set of limit points, and the end and starting points of such an
ordering are computable. 

A testament that \cref{upperbound} cannot be transferred to work with
progressions along ordinal notation systems is that it gives for every true $\Pi_1$
sentence $\phi$ a computable well-order $\mc L$ of order-type $1$ such that
$\mathsf{RFN}^{\mc L}(\mathsf{PA})\vdash \phi$. Ordinal notation systems such as
Kleene's $\mathcal O$ have a unique notation for finite orderings and thus the
analogous theorem for ordinal notation systems must fail. However, we can still
obtain tight bounds for progressions along ordinal notation systems by using
basic properties of computable linear orderings.

For an arbitrary number $a$ and c.e.\ axiomatizable $T$ extending $\mathsf{PA}$ we define $\mathsf{RFN}^a(T)$ to be
\begin{enumerate}
\item $\mathsf{RFN}^{2^a}(T)=\mathsf{RFN}(\mathsf{RFN}^a(T))$,
\item $\mathsf{RFN}^{3\cdot 5^e}(T)=T+\bigcup \{\mathsf{RFN}(\mathsf{RFN}^a(T))\mid a=\{e\}(n) \text{ for some }n\}$,
\item $\mathsf{RFN}^a(T)=T$, for $a$ that aren't of the forms $3\cdot 5^e$ or $2^{a'}$.
\end{enumerate}
We make this definition precise in a manner fully analogous to what we did in Section \ref{uniform_reflection_section}

\begin{theorem}\label{feferman-notations}
    For every true $\Pi_{2n+1}$ sentence $\phi$ there exists $a\in \mathcal
    O$ with $|a|=\omega^{n+1}+1$ such that
    \[ \mathsf{RFN}^a(\mathsf{PA}) \vdash \phi.\]
\end{theorem}
\begin{proof}
    Let $\phi$ be a $\Pi_{2n+1}$ sentence and $\mathcal{L}$ be the linear
    ordering obtained in \cref{sen2wo}. Then $\mathsf{ACA}_0$ proves $\mathsf{WO}(\mathcal L)$ if and
    only if $\phi$ is true. In particular, the ordering $\mathcal
    L'=\omega\cdot(1+\mathcal{L})+1$ is a computable ordering such that
    $\mathsf{ACA_0}$ proves $\mathcal L' \cong \omega^{n+1}+1$ if and only if
    $\phi$ is true. We can also assume that $\mathcal L'$ comes with other
    special properties shared with notations in Kleene's $\mathcal{O}$: It has a
    computable successor relation, a computable set of limit points, and there are
    algorithms that compute the first and last elements, predecessors, and fundamental
    sequences for limit points. For example, the ordering defined by
    \[\begin{split} L'&=\{\infty,(-1,m),(n,m):n\in L, m\in \omega\}, \\
        (n_0,m_0)\preceq^{\mathcal{L}'}
        (n_1,m_1) &\iff (n_0\preceq^{\mathcal{L}} n_1\lor n_0=-1) \land m_0\leq
    m_1, \\
\text{ and for all $x$, } x\preceq^{\mathcal{L}'} \infty \end{split}\]
clearly has these properties. It is shown in~\cite[Lemma 4.13]{ash2000} that
there is a computable function $g:L'\to \mathbb N$ such that for every $x\in
L'$,
if ${\preceq^{\mathcal{L}'}}x=\preceq^{\mathcal{L}'}\upharpoonright \{ y\preceq^{\mathcal{L}'} x: y\in
L'\}$ is well-ordered, then $g(x)\in \mathcal O$ and $|g(x)|\cong
{\preceq^{\mathcal{L}'}}x$. If ${\preceq^{\mathcal{L}'}}x$ is not well-ordered, then
$g(x)\not\in \mathcal{O}$. 

Using the same reasoning as in the proof of \cref{upperbound} we now get that if
$\phi$ is true, then $\mathsf{RFN}^{g(\infty)}(\mathsf{PA})\vdash \phi$, and
$|g(\infty)|=\omega^{n+1}+1$.
\end{proof}
\begin{corollary}[Feferman's original completeness theorem with improved bounds]
    For any true arithmetical sentence $\phi$, there exists $a\in \mathcal O$
    with $|a|<\omega^\omega$ such that $\mathsf{RFN}^{a}(\mathsf{PA})\vdash
    \phi$. Furthermore, the bound on $|a|$ is tight, i.e., for every
    $\alpha<\omega^\omega$, there is a true arithmetical sentence $\phi$ such
    that for no $a\in \mathcal O$ with $|a|=\alpha$,
    $\mathsf{RFN}^{a}(\mathsf{PA})\vdash \phi$.
\end{corollary}
A similar result was claimed in \cite{fenstad1968}; however, Savitt found a mistake in this proof \cite{savitt1969jens}.

\section{Discussion and Open Questions}
We suspect that it should be easy to
adopt the techniques from \cref{lowerbound} to show sharpness of the upper bound
from \cref{feferman-notations} on the order-types for the iterations of reflection along elements of $\mc O$ necessary to attain the completeness of iterations of reflection for $\Pi_{2n+1}$-sentences.

In the present paper, we have been focusing on the iterations of full uniform reflection,
however, when we want to obtain some true $\Pi_n$-sentence in principle it is
natural to expect that we need to iterate just partial uniform $\Pi_n$
reflection $\mathsf{RFN}_{\Pi_n}$ instead of full uniform reflection
$\mathsf{RFN}$. For example, the existence of such iterations could be easily
proved by combining Feferman's completeness theorem for full uniform
reflection with Schmerl's formula that allows us to transform short
iterations of stronger reflection to longer iterations of weaker $\Pi_n$
reflection, while preserving the set of $\Pi_n$-consequences. We suspect that the methods used in this paper can be adopted to show that for some natural
number constant $C$ and every $n$ we can replace $\mathsf{RFN}$ with
$\mathsf{RFN}_{\Pi_{2n+1+C}}$ in \cref{upperbound}. However, it is not clear to
us whether in \cref{upperbound} full uniform reflection $\mathsf{RFN}$ can be
replaced with just $\mathsf{RFN}_{\Pi_{2n+1}}$, while keeping the bound
$\omega^n+1$ on the order-type.


\printbibliography
\end{document}